\documentclass[11pt]{amsart}
\usepackage{amsmath}
\usepackage{amsthm}
\usepackage{hyperref,tikz}
\usepackage[margin=1.0in]{geometry}

\numberwithin{equation}{section}

\newtheorem{prop}{Proposition}
\newtheorem{lemma}[prop]{Lemma}

\newtheorem{thm}[prop]{Theorem}
\newtheorem{cor}[prop]{Corollary}

\numberwithin{prop}{section}

\theoremstyle{definition}
\newtheorem{defn}[prop]{Definition}

\newtheorem{rmk}[prop]{Remark}

\newcommand{\Fp}{F^{+}}
\newcommand{\Wp}{W^{+}}
\newcommand{\hg}{\hat{g}}

\newcommand{\brs}[1]{\left| #1 \right|}

\renewcommand{\gg}{\gamma}

\newcommand{\gD}{\Delta}
\newcommand{\gd}{\delta}

\newcommand{\gk}{\kappa}
\newcommand{\gl}{\lambda}

\newcommand{\gw}{\omega}
\newcommand{\ga}{\alpha}
\newcommand{\gb}{\beta}
\renewcommand{\ge}{\epsilon}
\newcommand{\N}{\nabla}

\renewcommand{\bar}[1]{\overline{#1}}

\newcommand{\bg}{\bar{g}}

\newcommand{\IP}[1]{\left<#1\right>}

\DeclareMathOperator{\tr}{tr}

\DeclareMathOperator{\Vol}{Vol}
\DeclareMathOperator{\SU}{SU}
\DeclareMathOperator{\SO}{SO}

\begin{document}

\title[A conformally invariant gap theorem in Yang--Mills theory]{A conformally
invariant gap theorem in Yang--Mills theory}

\author{Matthew Gursky}
\address{Department of Mathematics
         University of Notre Dame\\
         Notre Dame, IN 46556}
\email{\href{mailto:mgursky@nd.edu}{mgursky@nd.edu}}

\author{Casey Lynn Kelleher}
\address{Department of Mathematics
         Princeton University\\
         Princeton, NJ, 08540}
\email{\href{mailto:ckelleher@princeton.edu}{ckelleher@princeton.edu}}

\author{Jeffrey Streets}
\address{Department of Mathematics\\
         University of California\\
         Irvine, CA  92617}
\email{\href{mailto:jstreets@uci.edu}{jstreets@uci.edu}}

\date{\today}

\begin{abstract} We show a sharp conformally
invariant gap theorem for Yang--Mills connections in dimension $4$ by exploiting
an associated Yamabe-type problem.
\end{abstract}

\thanks{The first author acknowledges the support of NSF Grant DMS-1509633. The second author was supported by a University of California President's Dissertation Year Fellowship during this work. The third author acknowledges the support of NSF Grant DMS-1454854. }

\maketitle

\section{Introduction}
Let $(X^n, g)$ be a smooth Riemannian manifold, and suppose $\N$ is a
connection on a smooth vector bundle over $X$.  The Yang--Mills energy
associated to $\N$ is
\begin{align*}
\int_X \brs{F_{\N}}^2_g \, dV_g.
\end{align*}
Critical points of this functional are known as Yang--Mills connections, and
satisfy $D^*_{\N} F_{\N} = 0$.  In the case $n=4$, the Yang--Mills energy
admits a special class of critical points, namely those with (anti)self-dual
curvature, i.e. $\star F_{\N} = \pm F_{\N}$, known as instantons.  Instantons
are the key
ingredient in developing applications of Yang--Mills theory to four-dimensional
topology through Donaldson invariants (cf. \cite{Donaldson1,Donaldson2, DK}).
When they exist, instantons always have the minimum possible Yang--Mills energy.
 However, on many interesting bundles where the instantons are understood,
there exist non-minimizing Yang--Mills connections (e.g.
\cite{Bor,Parker,SS1,SSU}).  Moreover, many basic questions about the structure
of Yang--Mills connections beyond instantons remain unanswered, for instance on
the allowable energy levels.  We note that Bourguignon-Lawson (\cite{BL}
Theorems C, D) have
shown some gap results assuming pointwise smallness of some pieces of the
curvature, exploiting primarily the Bochner formula.  This was later improved
by Min-Oo (\cite{MinOo} Theorems 2, 4) to an $L^2$ gap theorem, exploiting an
$\ge$-regularity argument which requires positivity of a certain
curvature quantity associated to the background metric.  A closely related gap result in the presence of a curvature positivity condition was shown by Parker (\cite{Parker2} Proposition 2.2).  Later Feehan
(\cite{Feehan}) showed a more general $L^2$ gap theorem which removes this
positivity hypothesis, and again ultimately relies on $\ge$-regularity style
estimates, where the smallness of Yang-Mills energy is balanced against the Sobolev constant to obtain a key estimate. The main result of this paper improves upon all of these prior gap theorems.  In particular, we show a sharp, conformally invariant
improvement of these gap theorems which is nontrivial when the Yamabe invariant
$Y([g])$ of $(X^4,g)$ (see \cite{LeeParker}) is positive.

\begin{thm} \label{Thm1} Let $(X^4,g)$ be a closed, oriented four-manifold.
Suppose $\nabla$ is a Yang--Mills connection on  a vector bundle $E$ over
$X^4$ with structure group $G \subset \SO(E)$, and curvature $F_{\N}$.
Then one of the following must hold:
\begin{enumerate}
\item $F^{+}_{\N} \equiv 0$; or
\item $F^{+}_{\N}$ satisfies
\begin{align} \label{SG1}
Y([g]) \leq 3 \gamma_1 \| F^{+}_{\N} \|_{L^2} +  2 \sqrt{6} \| W^{+} \|_{L^2},
\end{align}
where $\gamma_1 = \gamma_1(E) \leq \frac{4}{\sqrt{6}}$ is a constant which
depends on the structure group of the bundle (see
Definition \ref{d:gammaconsts} below), and $W^+$ is the self-dual Weyl tensor.

Moreover, if equality holds in (\ref{SG1}) then $[g]$ admits a Yamabe metric
$\bg$ with respect to which $W^{+}$ has constant norm, $\nabla_{\bg}
F^{+}_{\nabla_{\bg}}
\equiv 0$, and
\begin{align} \label{equality1}
R_{\bg} - 2\sqrt{6} |W^{+}|_{\bg} = 3 \gg_1 |F^{+}_{\N}|_{\bg},
\end{align}
where $R_{\bg}$ is the scalar curvature of $\bg$.  Furthermore, if $\gamma_1 >
0$ we must have
$b_2^{+}(X^4) = 0$, whereas if $\gamma_1 = 0$ then all harmonic self-dual forms
are parallel.
\end{enumerate}
\end{thm}

The key idea to prove the inequality (\ref{SG1}) is to interpret a certain
B\"ochner estimate for $F^+_{\N}$ in light of the {\em modified Yamabe problem}
introduced in \cite{G1}.  This technique has been used in various
contexts to prove sharp $L^2$-curvature estimates under topological and
geometric assumptions (see \cite{GL1}, \cite{GL2}, \cite{LeBrun}).  In each of
these applications one studies a generalization of the Yamabe problem in which
the scalar curvature is modified by adding a conformal density of the correct
weight.  Carrying this method out requires a number of sharp linear algebraic
estimates carried out in \S \ref{ss:smi}, as well as a sharp improved Kato
inequality for bundle-valued differential forms shown in \S \ref{ss:Kato}.

The estimate (\ref{SG1}) is sharp, as illustrated in a key geometric
situation.  First, note that as follows from Lemmas
\ref{l:BLcommutatorestimate} and \ref{l:2formcommlemma}, one has the universal
bound $\gamma_1 \leq \frac{4}{\sqrt{6}}$.  Using this together with explicit
calculations for the Yamabe constant and the Yang--Mills energy as spelled out
in
\S \ref{ss:examples}, one sees that the classic example of $\SU(2)$ ADHM/BPST
instantons on
$(\mathbb{S}^4, g_{\mathbb{S}^4})$ (\cite{ADHM, BPST}) yields equality in
(\ref{SG1}).  Moreover, the existence of a conformally related metric for which
the curvature is parallel reflects the fact that all these instantons are
determined by the action of the conformal group on the unique $\SO(4)$-invariant ADHM/BPST
instanton (the ``standard'' ADHM/BPST
instanton) which has parallel curvature.

We note here that it may be possible to extend Theorem \ref{Thm1} to the case when the underlying Riemannian manifold is complete.  Previous results in this direction have been shown (\cite{MOD,Shen,Xin}), which rely on a curvature positivity condition (see also \cite{Gerhardt}).  As our proof relies on solving a kind of modified Yamabe problem, extending it to the complete setting requires knowledge of the asymptotics of the underlying metric and the given Yang-Mills connection.  This makes the sharp statement in this direction not completely clear, and we do not pursue this further here.

Theorem \ref{Thm1} has some immediate corollaries which we collect below.
First,
using the characteristic class formula from Chern-Weil theory we derive
a lower bound which we make explicit in the case of conformally flat metrics,
which includes of course the round sphere.  Here we take the convention
following (\cite{DK} (2.1.40)) for the meaning of the characteristic number
$\kappa(E)$.  Note in particular that $\kappa(E) = c_2(E)$ for $\SU(r)$
bundles.
However our metric convention differs from (\cite{DK}), see (\S
\ref{ss:examples}, (\ref{CW})).  To make things concrete, estimates
(\ref{SU2est}) and (\ref{SO3est}) below are saying for instance on a
$\brs{\gk(E)} = 1$ bundle,
any Yang--Mills connection which is not an instanton must have at least
\emph{three} units of charge, where intuitively one unit each of SD/ASD charge
cancel out to preserve the characteristic class condition.

\begin{cor} \label{cor1} Let $(X^4,g)$ be a closed, oriented, conformally flat
four-manifold with $Y([g]) > 0$.  Suppose $\nabla$ is a Yang--Mills connection
on
 a vector bundle $E$ over $X^4$ with structure group $G \subset \SO(E)$, and
curvature $F_{\N}$.  Then $\N$ is either an instanton, or satisfies
\begin{align} \label{energyineq}
\int_{X} \brs{F_{\N}}^2_g \, dV_g \geq 16 \pi^2 \brs{\gk(E)} +  \frac{2 Y([g])^2}{9
\gamma_1^2} \geq 16 \pi^2 \brs{\gk(E)} + \frac{Y([g])^2}{12}.
\end{align}
In particular,
\begin{enumerate}
\item For $E \to (\mathbb{S}^4, g_{\mathbb{S}^4})$ an $\SU(2)$ bundle, a
Yang--Mills connection $\N$
is either an instanton, or satisfies
\begin{align} \label{SU2est}
\int_{\mathbb{S}^4} \brs{F_{\N}}^2_g \, dV_g \geq&\ 16 \pi^2 \brs{\gk(E)} + 32 \pi^2.
\end{align}
\item For $E \to (\mathbb{S}^4, g_{\mathbb{S}^4})$ an $\SO(3)$ bundle, a
Yang--Mills connection $\N$
is either an instanton, or satisfies
\begin{align} \label{SO3est}
\int_{\mathbb{S}^4} \brs{F_{\N}}^2_g \, dV_g \geq&\ 16 \pi^2 \brs{\gk(E)} + 64 \pi^2.
\end{align}
\end{enumerate}
\end{cor}

Another application of Theorem \ref{Thm1} is to the Yang--Mills flow.
Fundamental work of Chen--Shen \cite{ChenShen}, Struwe, \cite{Struwe}, Schlatter \cite{Schlatter}, Kozono--Maeda--Naito \cite{KMN} shows that finite
time singularities of the Yang--Mills flow occur via energy concentration, and
moreover bubbling limits can be constructed which are Yang--Mills connections
over $(\mathbb{S}^4, g_{\mathbb{S}^4})$.  Although not explicitly stated, an
immediate corollary
of those works is a global existence and convergence statement for connections
over the trivial bundle with sufficiently small energy, where this constant depends on a gap theorem for connections on $\mathbb{S}^4$.  Thus Theorem \ref{Thm1} gives a statement about Yang--Mills flow which is much stronger than that previously attainable, yielding a computable universal
threshold below
which the flow must exist globally and converge.  In the case of $(\mathbb{S}^4, g_{\mathbb{S}^4})$ the limit is flat, and the corollary gives the sharp energy inequality which guarantees convergence to a flat connection.  Note
that
the hypothesis that the bundle is trivial is a necessary consequence of the
energy upper bound hypothesis, but we state it for clarity.

\begin{cor} \label{cor2} Let $(X^4,g)$ be a closed, oriented, four-manifold.
Suppose $E \to X^4$ is a trivial bundle with structure group $G \subset
\SO(E)$.
Suppose $\N$ is a connection on $E$ with curvature $F_{\N}$ satisfying
\begin{align} \label{flowbound}
\int_{X} \brs{F_{\N}}^2\, dV_g <&\ 16 \pi^2.
\end{align}
Then the solution to Yang--Mills flow with initial condition $\N$ exists on
$[0,\infty)$, and converges as $t \to \infty$ to a Yang--Mills connection.  If $(X^4, g) \cong (\mathbb{S}^4, g_{\mathbb{S}^4})$ then the limit is flat.
\end{cor}

\noindent \textbf{Acknowledgements:} The authors thank Paul Feehan, Alex Waldron, and referees for helpful comments.

\section{Preliminary inequalities}

In this section we establish preliminary inequalities necessary for the proof
of
Theorem \ref{Thm1}.
In particular, in \S \ref{ss:smi}
we show sharp matrix inequalities needed to estimate nonlinear terms
which arise.  We also establish an improved Kato
inequality for Lie algebra valued $2$-forms in \S \ref{ss:Kato} which yields
favorable inequalities for small powers of $\brs{F^+_{\N}}$.  To
set the stage, we restate the fundamental B\"ochner formula here for
convenience.  First let us fix certain conventions.  For sections of
$g_{E}$ we define the canonical inner product
\begin{align} \label{f:endoIP}
\IP{A,B} := - \tfrac{1}{2} \tr(AB).
\end{align}
The factor of $\tfrac{1}{2}$ is in keeping with the convention of
Bourguignon-Lawson (\cite{BL} (2.14)).  This inner product is positive definite
since the only Lie algebras we consider satisfy $\mathfrak g_E \subset
\mathfrak{so}_E$, the space of skew symmetric endomorphisms of $E$.  In the
Bochner formula below, the inner products use the given Riemannian metric in
conjunction with (\ref{f:endoIP}).

\begin{lemma} \label{l:Bochner} (\cite{BW}, \cite{BL} Theorem 3.10) Let $(X^4, g)$ be a Riemannian manifold, and
suppose $E \to X$ is a smooth vector bundle with connection $\N$.  If $\gw \in
\Lambda_+^2(\mathfrak g_E)$ is a harmonic two-form, one has
\begin{align} \label{f:genBochner}
\tfrac{1}{2} \Delta |\gw|^2 = |\nabla \gw|^2 - \langle \gw, [
F^{+},
\gw ] \rangle - 2 \langle \gw, W^{+} \star \gw \rangle +
\tfrac{1}{3} R \brs{\gw}^2,
\end{align}
where with respect to local bases one has, for $P,Q \in \Lambda^2_+(\mathfrak
g_E)$, the tensor $[P,Q] \in \Lambda^2_+(\mathfrak g_E)$ defined via
\begin{align} \label{f:2formcommutator}
[P,Q]_{ij \ga}^{\gb} := g^{kl} \left( P_{ik \gd}^{\gb} Q_{j l \ga}^{\gd} - P_{i
k \ga}^{\gd} Q_{j l \gd}^{\gb} - P_{jk \gd}^{\gb} Q_{i l \ga}^{\gd} + P_{j k
\ga}^{\gd} Q_{i l \gd}^{\gb} \right),
\end{align}
and
\begin{align*}
(W^+ \star \gw)_{ij \ga}^{\gb} :=&\ g^{kp} g^{lq} W^+_{ijkl} \gw_{pq
\ga}^{\gb}.
\end{align*}
\end{lemma}

\subsection{Sharp matrix inequalities} \label{ss:smi}

In this subsection we establish estimates for the curvature and
covariant derivative terms in (\ref{f:genBochner}).  First we estimate the Weyl
curvature term.  As the action is induced from the natural action on real
valued self-dual two-forms, the proof is a straightforward modification of that
case.

\begin{lemma} \label{sharpW} Let $(X^4, g)$ be a Riemannian manifold, and
suppose $E \to X$ is a smooth vector bundle with connection $\N$.  If $\gw \in
\Lambda_+^2({\mathfrak g}_E)$ is a two-form, one has
\begin{align} \label{WFF}
\brs{ \langle \gw, W^{+} \star \gw \rangle} \leq \tfrac{2}{\sqrt{6}}
|\Wp| |\gw|^2.
 \end{align}
 \end{lemma}
\begin{proof} This follows directly from the fact that $W^+$ is a trace-free
endomorphism of $\Lambda^2_+$, a rank $3$ vector bundle (cf. \cite{SW} p 234).
\end{proof}

The estimate of the bracket term in (\ref{f:genBochner}) will depend on the
Lie
algebra.  We
define two constants relevant to understanding the bracket term in the
B\"ochner
formula.

\begin{defn} \label{d:gammaconsts} Given $\mathfrak g \subset \mathfrak
{so}_E$, let
\begin{align*}
\gamma_{0} :=&\ \sup_{A,B \in \mathfrak g \backslash \{0\}}
\frac{\brs{[A,B]}}{\brs{A} \brs{B}},\\
\gamma_1 :=&\ \sup_{\omega \in \Lambda^2_{+}(\mathfrak{g}) \setminus \{0\}}
\dfrac{ \langle \omega, [\omega, \omega]\rangle }{|\omega|^3}.
\end{align*}
\end{defn}

These suprema are certainly attained as the quantities are scale invariant and
defined on finite dimensional vector spaces.  A crucial point however is that
these quantities depend on the choice of metric on
$\mathfrak g_E$.  Thus the choice (\ref{f:endoIP}) is important to what
follows, and fixes the ``scale'' of various terms involving Lie algebra values.
 A fundamental lemma of Bourguignon-Lawson gives a universal upper bound for
$\gg_0$, which is further improvable for some special Lie algebras.

\begin{lemma} \label{l:BLcommutatorestimate} (\cite{BL} Lemma 2.30) Given
$\mathfrak g \subset \mathfrak{so}_E$, one has $\gg_{0} \leq \sqrt{2}$, with
the supremum defining $\gg_{0}$ attained by pairs $A, B$ which are
simultaneously equivalent to a pair of Pauli matrices, i.e.
\begin{align}\label{eq:Pauli}
\newcommand*{\temp}{\multicolumn{1}{c|}{0}}
A =
\left(\begin{array}{c|c}
\begin{array}{c|c}
  \begin{matrix} 0 & t \\ -t & 0\end{matrix}  &0\\ \cline{1-2}
    0 &  \begin{matrix} 0 & t \\ -t & 0\end{matrix} \\
   \end{array} & 0   \\ \cline{1-2}  
    \temp & 0
\end{array}\right), \qquad
B = \left(\begin{array}{c|c}
\begin{array}{c|c}
0  &  \begin{matrix} -s & 0 \\ 0 & s\end{matrix}\\ \cline{1-2}
 \begin{matrix} s & 0 \\ 0 & -s\end{matrix} &  0\\
   \end{array} & 0   \\ \cline{1-2}  
 \temp & 0
\end{array}\right).\end{align}
Furthermore, in the case $\mathfrak g = \mathfrak {so}_3$ one has $\gg_0 = 1$.
\end{lemma}

Next, using Lemma \ref{l:BLcommutatorestimate}, we establish an estimate for
the wedge product commutator of Lie algebra valued $2$-forms.  We first record an
elementary linear algebra fact.  Note here that we take the metric induced on
two-forms as
\begin{align*}
 \IP{\eta,\mu} := g^{ik} g^{jl} \eta_{ij} \mu_{kl}.
\end{align*}

\begin{lemma} \label{l:circbasis} Given $(V^4, g)$ an inner product space, for
$\eta,\mu \in \Lambda^2_+(V^*)$ define $(\eta \circ \mu) \in \Lambda^2_+(V^*)$
via
\begin{align*}
(\eta \circ \mu)_{ij} = g^{kl} \left( \eta_{ik} \mu_{jl} - \eta_{jk} \mu_{il}
\right).
\end{align*}
Then given $\{e_1,e_2,e_3\}$ an orthonormal basis for $\Lambda^2_+(V^*)$, one
has that $\{ (e_1 \circ e_2), (e_1 \circ e_3), (e_2 \circ e_3) \}$ is another
orthonormal basis.
\end{lemma}

\begin{lemma} \label{l:2formcommlemma} Let $E \to (M^4,g)$ be a smooth vector
bundle.  Given $P \in \Lambda^2_+(\mathfrak g_E)$, one has
\begin{align*}
\brs{[P,P]} \leq \tfrac{2}{\sqrt{3}} \gg_{0} \brs{P}^2.
\end{align*}
\begin{proof} Fix a point $p \in M$ and let $\{e_i\}, i = 1,2,3$ be an
orthonormal basis for $\Lambda^2_+$ at $p$.  Now let us express $P = P^i e_i$,
where $P^i \in \mathfrak g_E$.  Due to the skew commutativity of the bracket
structure we may estimate
\begin{align}\label{eq:Pbrack}
\begin{split}
\brs{[P,P]}^2 =&\ \brs{[ P^i e_i, P^j e_j]}^2\\
=&\ \brs{ 2(e_1 \circ e_2) \otimes [P^1,P^2] + 2 (e_1 \circ e_3) \otimes
[P^1,P^3] + 2 (e_2 \circ e_3) \otimes [P^2,P^3]}^2\\
\leq&\ 4 \left( \brs{[ P^1,P^2]}^2 + \brs{[P^1,P^3]}^2 + \brs{[P^2,P^3]}^2
\right)\\
\leq&\ 4 \gg_0^2 \left( \brs{P^1}^2 \brs{P^2}^2 + \brs{P^1}^2 \brs{P^3}^3 +
\brs{P^2}^2 \brs{P^3}^2 \right)\\
\leq&\ \tfrac{4 \gg_0^2}{3} \left( \brs{P^1}^2 + \brs{P^2}^2 + \brs{P^3}^2
\right)^2\\
=&\ \tfrac{4 \gg_0^2}{3} \brs{P}^4.
\end{split}
\end{align}
Taking the square root yields the claim.
\end{proof}
\end{lemma}

\begin{rmk} \label{r:BPST} We note that equality in Lemma
\ref{l:2formcommlemma} is achieved by the curvature of the standard ADHM/BPST
instanton.  At a fixed point this tensor takes the form (cf. \eqref{f:BPSTcurv})
\begin{align*}
F_{\N} = \gl \left\{ \left( dx^{12} + dx^{34}
\right) \otimes \textbf{i} + \left( dx^{13} - dx^{24} \right) \otimes
\textbf{j}
+ \left( dx^{14}
+ dx^{23} \right) \otimes \textbf{k} \right\}.
\end{align*}
Note that, in the notation of Lemma \ref{l:2formcommlemma}, this curvature is
expressed as $e_i P^i = e_1 \textbf{i} + e_2 \textbf{j} + e_3 \textbf{k}$ for
$\{e_1,e_2,e_3\}$ the standard basis for $\Lambda^2_+$.  Using the quaternion
relations, it is clear that the pairwise commutators between the $P^i$ are thus
orthogonal, making the third line of \eqref{eq:Pbrack} an equality.  These matrices $P^i$ are Pauli matrices,
so by Lemma \ref{l:BLcommutatorestimate} the fourth line is also an equality.
Lastly, as each $P^i$ has the same norm, the fifth line of \eqref{eq:Pbrack} is an equality.
\end{rmk}

\subsection{Improved Kato inequality} \label{ss:Kato}

In this subsection we prove a sharp Kato inequality for Lie algebra valued
harmonic two-forms on four-manifolds.  This was proved for Yang--Mills
connections on $\mathbb{R}^4$ in \cite{Rade}.  Our proof is an elementary
modification of the method of Seaman \cite{Seaman}, who showed a sharp Kato
inequality for harmonic real valued two-forms on four-manifolds, and exploited
it to derive vanishing results for positively curved four-manifolds.  Seaman's
method exploits the conformal invariance of harmonic two-forms in four
dimensions, together with a delicate comparison of the B\"ochner formula for
two
choices of conformal factor.  The Yang--Mills equation is also conformally
invariant in four-dimensions, and the relevant Bochner formula only differs
by
a conformally invariant term, the proof is adapted in a straightforward manner.

\begin{prop} \label{sharpKato} Let $E \to (X^4, g)$ be a vector bundle over a
smooth Riemannian four-manifold.  Given $\N$ a connection on $E$, and $\gw \in
\Lambda^2(\mathfrak g_E)$ a harmonic two-form, one has the pointwise inequality
\begin{align} \label{SK}
\brs{\N \gw}^2 \geq \tfrac{3}{2} \brs{d \brs{\gw}}^2.
\end{align}
\begin{proof} First, by applying the Bochner formula (Lemma \ref{l:Bochner})
to
$\gw$ (summing over self-dual and antiself-dual parts) we obtain
\begin{align} \label{f:sharpKato10}
\tfrac{1}{2} \gD \brs{\gw}^2 = \brs{\N \gw}^2 + \IP{[F,\gw] - 2 W \star \gw +
\tfrac{R}{3} \gw,\gw}.
\end{align}
We now make a conformal modification of the metric and derive a second
B\"ochner
identity.  In particular, let $\hg = \brs{\gw} g$, which defines a smooth
Riemannian metric away from the zero locus of $\brs{\gw}$.  Note by
construction that $\brs{\gw}^2_{\hg} \equiv 1$.  Furthermore, as the condition
that $\gw$ is harmonic is conformally invariant, $\gw$ is harmonic with respect
to $\hg$, and thus we apply the Bochner formula again to conclude
\begin{align} \label{f:sharpKato20}
0 =&\ \tfrac{1}{2} \hat{\gD} \brs{\gw}^2_{\hg} = \brs{\hat{\N} \gw}^2_{\hg} +
\IP{[\hat{F},\gw] - 2 \hat{W} \star \gw + \tfrac{\hat{R}}{3} \gw,\gw}_{\hg}.
\end{align}
The bundle curvature $F$ and the Weyl tensor $W$ are conformally covariant,
yielding
\begin{align} \label{f:sharpKato30}
\IP{[\hat{F},\gw] - 2 \hat{W} \star \gw,\gw}_{\hg} = \brs{\gw}^{-3}
\IP{[F,\gw] - 2 W \star \gw,\gw}.
\end{align}
On the other hand, using the transformation formula for the scalar curvature
under conformal change one obtains
\begin{align} \label{f:sharpKato40}
\IP{\tfrac{\hat{R}}{3} \gw, \gw}_{\hg} = \brs{\gw}^{-3} \left(
\IP{\tfrac{R}{3} \gw,\gw} - \tfrac{1}{2} \gD \brs{\gw}^2 + \tfrac{3}{2} \brs{d
\brs{\gw}}^2 \right).
\end{align}
Plugging (\ref{f:sharpKato30}) and (\ref{f:sharpKato40}) into
(\ref{f:sharpKato20}), and incorporating (\ref{f:sharpKato10}) we conclude
\begin{align*}
\brs{\hat{\N} \gw}^2_{\hg} =&\ - \brs{\gw}^{-3} \left( \IP{[F,\gw] + W
\star \gw + \tfrac{R}{6} \gw,\gw} - \tfrac{1}{2} \gD \brs{\gw}^2 + \tfrac{3}{2}
\brs{d \brs{\gw}}^2 \right)\\
=&\ \brs{\gw}^{-3} \left( \brs{\N \gw}^2 - \tfrac{3}{2} \brs{d
\brs{\gw}}^2 \right).
\end{align*}
This implies the desired inequality away from the vanishing locus of
$\brs{\gw}$, which in turn implies the inequality at all points.
\end{proof}
\end{prop}

\section{Main Proofs}

In this section we give the proof of Theorem \ref{Thm1}.  As discussed in the
introduction the proof involves a delicate application of ideas from conformal
geometry to the B\"ochner formula for $F^+_{\N}$.  The proof of Theorem
\ref{Thm1}
appears in \S \ref{ss:mainargument}.
Then in \S
\ref{ss:examples} we give an example illustrating the sharpness of the
estimate.
 We conclude in \S \ref{ss:corollaries} with the proofs of Corollaries
\ref{cor1} and \ref{cor2}.

\subsection{Proof of Theorem \ref{Thm1}} \label{ss:mainargument}

\begin{proof} Let $\N$ denote a Yang--Mills connection, and let us set $F =
F_{\N}$ for convenience.  Then $DF = 0, D^* F =
0$, and it follows easily that $\star F$ is also closed and co-closed,
hence $D
F^+ = 0$, $D^* F^+ = 0$.  Thus we may apply Lemma
\ref{l:Bochner} to $F^+$ to
yield
\begin{align} \label{Bochner}
\tfrac{1}{2} \Delta |F^{+}|^2 = |\nabla F^{+}|^2 - \langle F^{+}, [ F^{+},
F^{+}
] \rangle - 2 \langle F^{+}, W^{+} \star F^{+} \rangle + \tfrac{1}{3} R |\Fp|^2.
\end{align}
By the result of Lemma \ref{sharpW} and the definition of $\gamma_1$, we
immediately obtain
\begin{align} \label{B2}
\tfrac{1}{2} \Delta |F^{+}|^2 \geq |\nabla F^{+}|^2 + \tfrac{1}{3} ( R - 2
\sqrt{6} |\Wp| - 3 \gamma_1 |\Fp|) |\Fp|^2.
\end{align}
By the Leibniz rule, away from the zero locus of $\Fp$ we have
\begin{align*}
\tfrac{1}{2} \Delta |\Fp|^2 = |\Fp| \Delta |\Fp| + | \nabla |\Fp||^2,
\end{align*}
hence by (\ref{SK})
\begin{align}  \label{B2p}
\Delta |\Fp| \geq \tfrac{1}{2}\frac{ |\nabla |\Fp||^2}{|\Fp|} + \tfrac{1}{3} ( R
- 2 \sqrt{6} |\Wp| - 3 \gamma_1 |\Fp|) |\Fp|.
\end{align}
It follows that
\begin{align} \label{Ytest}
\Delta |\Fp|^{1/2} \geq \tfrac{1}{6} \big(  R - 2 \sqrt{6} |\Wp| - 3 \gamma_1
|\Fp| \big) |\Fp|^{1/2}
\end{align}
off the zero locus of $\Fp$, or in the sense of distributions.

We next exploit (\ref{Ytest}) in
conjunction with as a modified Yamabe problem introduced in \cite{G1}.  In
particular, given a metric $\hg \in [g]$ we define
\begin{align}  \label{Phidef}
\Phi_{\hg} = R_{\hg}  - 2 \sqrt{6} |\Wp|_{\hg} - 3 \gamma_1 |\Fp|_{\hg}.
\end{align}
If we drop the Weyl and $\Fp$-terms then $\Phi$ is just the scalar curvature.
Moreover, because these terms transform by scalings of the same weight under
conformal changes of metric, the transformation law for $\Phi$ is essentially
the same as the scalar curvature.  More precisely, if we define the natural
generalization of the conformal Laplacian by
\begin{align*}
L = - 6 \Delta + \Phi,
\end{align*}
then given $\hg = u^2 g$ it follows that
\begin{align} \label{PhiChange}
\Phi_{\hg} = u^{-3} L_g u.
\end{align}
Moreover, $L$ is conformally covariant:
\begin{align*}
L_{\hg} \phi = u^{-3} L_g (u \phi).
\end{align*}
Consequently, if $\lambda_1(L)$ denotes the first eigenvalue of $L$,
\begin{align} \label{LL}
\lambda_1(L_g) = \inf_{\phi \in C^{\infty}, \phi \neq 0} \dfrac{ \int_X \phi L_g
\phi\ dV_g }{\int_X \phi^2\ dV_g},
\end{align}
then the sign of $\lambda_1(g)$ is a conformal invariant.  In particular, by
using an eigenfunction associated with $\lambda_1(L)$ as a conformal factor, it
follows that $[g]$ admits a metric $\hg$ with $\Phi_{\hg} > 0$ ({\em resp.}, $=
0, < 0$) if and only if $\lambda_1(L_g) > 0$ ({\em resp.} $= 0, < 0$).

One departure from the classical Yamabe problem is that the modified scalar
curvature may only be Lipschitz continuous.  Therefore, the Schauder estimates
imply that the first eigenfunction is in $C^{2,\alpha}$ and hence defines a
conformal metric which is only $C^{2,\alpha}$.  One can smooth $|\Fp|$ and
approximate (see Section 3 of \cite{G1} for details), but in our setting this
will not be necessary.

Returning to the inequality (\ref{Ytest}), we can now express this as
\begin{align} \label{Ytest2}
0 \geq L_g ( |\Fp|^{1/2}).
\end{align}
Multiplying by $|\Fp|^{1/2}$ and integrating over $X^4$ gives
\begin{align*}
0 \geq \int_X |\Fp|^{1/2} L_g ( |\Fp|^{1/2})\, dV_g.
\end{align*}

It thus follows that either $F^+ \equiv 0$ or $\lambda_1(L_g) \leq 0$.  The
case $F^+ \equiv 0$ is case (1) of the statement, thus we proceed to analyze
the case $\gl_1(L_g) \leq 0$.  Let $\phi_1 > 0$ denote an eigenfunction
associated to $\lambda_1(L)$, and define the metric $\bar{g} = \phi_1^2 g$.  By
(\ref{PhiChange}),
\begin{align*}
\Phi_{\bg} = \phi_1^{-3} L_g \phi_1 = \lambda_1 \phi_1^{-2} \leq 0.
\end{align*}
Therefore,
\begin{align*}
0 \geq \int_X \Phi_{\bg} \, dV_{\bg} = \int_X \big( R_{\bg} - 2\sqrt{6} |\Wp| -
3
\gamma_1 |\Fp|_{\bg} \big) \, dV_{\bg},
\end{align*}
or
\begin{align} \label{L1}
\int_X R_{\bg} \, dV_{\bg} \leq 2 \sqrt{6} \int_X |\Wp|_{\bg} \, dV_{\bg} + 3
\gamma_1 \int_X |\Fp|_{\bg} \, dV_{\bg}.
\end{align}
We can estimate the integral on the left-hand side in terms of the Yamabe
invariant of $[g]$:
\begin{align} \label{Ybelow}
\int_X R_{\bg} \, dV_{\bg} \geq Y([g]) \Vol(\bg)^{1/2}.
\end{align}
For the terms on the right-hand side of (\ref{L1}) we use Cauchy-Schwartz:
\begin{align} \label{CS} \begin{split}
2 \sqrt{6} \int_X |\Wp|_{\bg} \, dV_{\bg}& + 3 \gamma_1 \int_X |\Fp|_{\bg} \,
dV_{\bg} \\
&\leq 2
\sqrt{6} \big( \int_X |\Wp|_{\bg}^2 \, dV_{\bg} \big)^{1/2} \Vol(\bg)^{1/2} + 3
\gamma_1 \big( \int_X |\Fp|_{\bg}^2 \, dV_{\bg} \big)^{1/2} \Vol(\bg)^{1/2} \\
&= 2 \sqrt{6} \big( \int_X |\Wp|_{g}^2 \, dV_{g} \big)^{1/2} \Vol(\bg)^{1/2} +
3
\gamma_1 \big( \int_X |\Fp|_{g}^2 \, dV_g \big)^{1/2} \Vol(\bg)^{1/2},
\end{split}
\end{align}
where the second line follows from conformal invariance of the integrals.
Combining (\ref{L1})-(\ref{CS}) and dividing by the square root of the volume we
arrive at (\ref{SG1}).

If equality is achieved then all of the inequalities above become equalities.
Equality in (\ref{Ybelow}) implies that $\bg$ is a Yamabe metric (hence
$C^{\infty})$, and
\begin{align} \label{FBpos}
R_{\bg} - 2\sqrt{6} |\Wp|_{\bg} = 3\gamma_1 |\Fp|_{\bg},
\end{align}
which proves (\ref{equality1}).  Since equality is attained in (\ref{CS}), it
follows that both $|\Wp|_{\bg}$ and $|\Fp|_{\bg}$ are constant.  By conformal
invariance of
the Yang--Mills energy, we can write the B\"ochner formula (\ref{Bochner}) with
respect to any metric in $[g]$.  If we use $\bg$ in place of $g$, then
(\ref{B2}) becomes
\begin{align} \label{B2B} \begin{split}
0 = \tfrac{1}{2} \Delta_{\bg} |F^{+}|_{\bg}^2 \geq |\nabla F^{+}|_{\bg}^2 +
\tfrac{1}{3} \Phi_{\bg} |\Fp|_{\bg}^2 = |\nabla F^{+}|_{\bg}^2,
\end{split}
\end{align}
hence $\Fp$ is parallel.

Finally, we establish the statements concerning $b_2^+$.  By the Bochner
formula for
real-valued self-dual two-forms (a special case of (\ref{Bochner})),
\begin{align} \label{WFb2}
\tfrac{1}{2} \Delta_{\bg} |\omega |^2 = |\nabla \omega|^2 - 2 \Wp_{\bg}
(\omega,
\omega) + \tfrac{1}{3} R_{\bg} |\omega|^2.
\end{align}
As a special case of Lemma \ref{sharpW} we have
\begin{align*}
|\Wp_{\bg} (\omega, \omega)| \leq \tfrac{2}{\sqrt{3}} |\Wp||\omega|^2.
\end{align*}
Therefore,
\begin{align*}
\tfrac{1}{2} \Delta_{\bg} |\omega |^2 &= |\nabla \omega|^2 - 2 \Wp_{\bg}
(\omega, \omega) + \tfrac{1}{3} R_{\bg} |\omega|^2 \\
&\geq |\nabla \omega|^2 + \tfrac{1}{3} \big( R_{\bg} - 2\sqrt{6} |\Wp|_{\bg}
\big) |\omega|^2 \\
&=|\nabla \omega|^2 + \gamma_1 |\Fp|_{\bg} |\omega|^2,
\end{align*}
where the last line follows from (\ref{FBpos}).  Since $|\Fp|_{\bar{g}}$ is
constant and
non-zero by assumption, we see that $\omega$ must vanish if $\gamma_1 > 0$,
implying $b_2^{+} = 0$.  If $\gamma_1 = 0$ we see that $\gw$ must be parallel,
as claimed.
\end{proof}

\subsection{Sharpness via \texorpdfstring{\boldmath$\SU(2)$}{SU(2)} instantons}
\label{ss:examples}

It is possible to achieve equality in (\ref{SG1}) via the classic construction
of BPST/ADHM $\SU(2)$ instantons on $\mathbb{S}^4$ (\cite{ADHM, BPST}).  We
recall the
most basic connection
in this class, the standard connection, for convenience (and also as a way to fix conventions). The standard connection is expressed on $\mathbb R^4$, thought of as $\mathbb{S}^4 \backslash \{N\}$) as the
$\mathfrak{su}(2)$-valued
$1$-form (cf. \cite{DK} (3.4.2))
\begin{align*}
\theta = \tfrac{1}{1 + \brs{x}^2} \left(\theta_1 \otimes \textbf{i} + \theta_2
\otimes \textbf{j}
+ \theta_3 \otimes \textbf{k} \right),
\end{align*}
where,
\begin{align*}
\theta_1 =&\ x_1 dx^2 - x_2 dx^1 + x_3 dx^4 - x_4 dx^3\\
\theta_2 =&\ x_1 dx^3 - x_3 dx^1 + x_4 dx^2 - x_2 dx^4\\
\theta_3 =&\ x_1 dx^4 - x_4 dx^1 + x_2 dx^3 - x_3 dx^2.
\end{align*}
Also, in keeping with our previous convention, we represent
$\{\textbf{i},\textbf{j}, \textbf{k}\}$ as real
matrices
\begin{align} \label{f:su2pauli}
\textbf{i} =&\
 \left(\begin{array}{c|c}
0  &  \begin{matrix} 1 & 0 \\ 0 & 1\end{matrix}\\ \cline{1-2}
\begin{matrix} -1 & 0 \\ 0 & -1\end{matrix} &  0\\
\end{array}\right),
 \quad \textbf{j} =
  \left(\begin{array}{c|c}
0  &  \begin{matrix} 0 & 1 \\ -1 & 0\end{matrix}\\ \cline{1-2}
\begin{matrix} 0 & 1 \\ -1 & 0\end{matrix} &  0\\
\end{array}\right)
, \quad \textbf{k} =
  \left(\begin{array}{c|c}
\begin{matrix} 0 & 1 \\ -1 & 0\end{matrix}  & 0 \\ \cline{1-2}
0 & \begin{matrix}  0 & -1 \\ 1 & 0\end{matrix}\\
\end{array}\right)
.
\end{align}
A classic calculation yields the relevant curvature tensor
\begin{align} \label{f:BPSTcurv}
F_{\N} = \tfrac{2}{\left(1 + \brs{x}^2\right)^2} \left\{ \left( dx^{12} +
dx^{34}
\right) \otimes \textbf{i} + \left( dx^{13} - dx^{24} \right) \otimes
\textbf{j}
+ \left( dx^{14}
+ dx^{23} \right) \otimes \textbf{k} \right\}.
\end{align}
Using our conventions for the metric induced by the Euclidean metric on
two-forms, $\brs{dx^{12}}^2 = 2$, etc.  Moreover, using (\ref{f:su2pauli}) and
(\ref{f:endoIP}), it follows that $\brs{\textbf{i}}^2 = \brs{\textbf{j}}^2 =
\brs{\textbf{k}}^2 = 2$.
Putting these together we can compute the pointwise norm of $F_{\N}$ in this context
to yield
\begin{align*}
\brs{F_{\N}}^2 =&\ \tfrac{4}{\left(1 + \brs{x}^2 \right)^4} \left(2
\brs{dx^{12} +
dx^{34}}^2 + 2 \brs{dx^{13} - dx^{24}}^2 + 2 \brs{dx^{14} + dx^{23}}^2 \right)\\
=&\ \tfrac{96}{\left(1 + \brs{x}^2 \right)^4}.
\end{align*}
Using the conformal invariance of the Yang--Mills energy, we thus obtain
\begin{align*}
\brs{\brs{F_{\N}}}_{L^2(\mathbb{S}^4, g_{\mathbb{S}^4})}^2 =&\ 96 \int_{\mathbb R^4}
\tfrac{1}{
\left(1
+ \brs{x}^2 \right)^4} \, dV_{Eucl} = 16 \pi^2.
\end{align*}

Having computed the Yang--Mills energy, we turn to the remaining quantities in
(\ref{SG1}).  As follows from \cite{Aubin}, the Yamabe
invariant of the round
$4$-sphere is $4(3) (\gw_4)^{\frac{1}{2}} = 12 \left( \frac{8}{3} \pi^2
\right)^{\frac{1}{2}} = 8 \sqrt{6} \pi$.  Thus, employing Lemmas
\ref{l:BLcommutatorestimate} and \ref{l:2formcommlemma} we observe that the
right hand side of the right hand side of (\ref{SG1}) can be estimated as
\begin{align*}
3 \gg_1 \brs{\brs{F^+_{\N}}}_{L^2} \leq 3 \left( \tfrac{2}{\sqrt{3}} \gg_0
\right)
(4
\pi) \leq 8 \sqrt{6} \pi = Y([g_{\mathbb{S}^4}]).
\end{align*}
Thus (\ref{SG1}) is an equality, as are all the intermediate estimates.  These
equalities reflect many interesting geometric properties of this classic charge
$1$ SU(2) instanton.  First, as discussed in Remark \ref{r:BPST}, the curvature
of this connection gives equality in the relevant algebraic inequalities we
used.  Furthermore, the theorem yields parallelism of $F^+_{\N}$ with respect
to
a
particular representative of $[g]$.  One can directly compute that the
distinguished BPST/ADHM connection representing the center of the moduli space
has parallel curvature with respect to the round metric.  However, since all
BPST/ADHM connections are given by pullback by an element of the conformal
group, one immediately concludes again that any such connection has parallel
curvature with respect to a particularly chosen element of the conformal class,
which our method explicitly constructs in a more general fashion via the
solution to the modified Yamabe problem.

\subsection{Proofs of Corollaries}\label{ss:corollaries}

\begin{proof}[Proof of Corollary \ref{cor1}] We recall the fundamental
Chern-Weil formula
\begin{align} \label{CW}
16 \pi^2 \kappa(E) = \int_X \tr (F_{\N} \wedge F_{\N}) = \int_X \left(
\brs{F_{\N}^-}_g^2 - \brs{F_{\N}^+}_g^2 \right) \, dV_g.
\end{align}
Let us first assume $\gk(E) \geq 0$, and $F_{\N}^+ \neq 0$, with the case
$\kappa(E) \leq 0$ directly analogous.  Since we have assumed our metric is
conformally flat and $Y([g]) > 0$, we may combine (\ref{SG1}) and (\ref{CW}) to
yield
\begin{align*}
\brs{\brs{F_{\N}}}_{L^2}^2 = \brs{\brs{F^-_{\N}}}_{L^2}^2 +
\brs{\brs{F_{\N}^+}}_{L^2}^2 = 16 \pi^2 \brs{\gk(E)} + 2 \int_X \brs{F_{\N}^+}^2_g  \, dV_g
\geq 16 \pi^2 \brs{\gk(E)} +  \tfrac{2 Y([g])^2}{9 \gamma_1^2},
\end{align*}
as claimed.  For the special case of $(\mathbb{S}^4, g_{\mathbb{S}^4})$, as
discussed in \S
\ref{ss:examples} know that the Yamabe invariant is $8 \sqrt{6} \pi$.  Moreover
using Lemmas \ref{l:BLcommutatorestimate} and \ref{l:2formcommlemma}, we know
that in the case of structure group $\SU(2)$ we may choose $\gamma_1 =
\frac{4}{\sqrt{6}}$, whereas for $\SO(3)$ one has $\gamma_1 =
\frac{2}{\sqrt{3}}$.  This yields the remaining statements.
\end{proof}

\begin{proof}[Proof of Corollary \ref{cor2}] We give a very brief sketch which
assumes familiarity with the papers (\cite{Schlatter}, \cite{Struwe}), and
Yang--Mills flow in general.  In particular, as discussed in (\cite{Struwe}
Theorem 2.3, \cite{Schlatter} Theorem 1.1), any  smooth connection admits a
unique solution to Yang--Mills flow with the prescribed initial data, which
moreover encounters a singularity at either finite or infinite time via
``concentration of energy.''  As made precise in (\cite{Struwe} Theorem 2.4,
\cite{Schlatter} Theorem 1.2), at any singular point in spacetime one can
construct at least a maximal bubble defined as a limit of blowup sequences,
which converge in the Uhlenbeck sense to a nontrivial Yang--Mills connection
over
$(\mathbb{S}^4, g_{\mathbb{S}^4})$.  Crucially, the energy of this limiting
connection is no
larger than the energy of the initial connection.  Thus, comparing
(\ref{flowbound}) against the results in Corollary \ref{cor1}, we see that the
energy inequalities cannot hold, and therefore this limiting connection must be
an instanton.  However, comparing against (\ref{CW}), we see that any nonflat
instanton must have energy at least $16 \pi^2$, thus we have arrived at a
contradiction.  Thus the flow exists globally and the time slices converge
subsequentially as time approaches infinity to a smooth limiting Yang--Mills
connection.  In the case of $(\mathbb{S}^4, g_{\mathbb{S}^4})$ it is clear by the argument above that the limiting connection is flat.
By employing \L{}ojasiewicz-Simon arguments (cf. \cite{Rade2} Proposition 7.2, \cite{Yang}, \cite{FeehanBook} Theorem 7, all based on the classic \cite{Simon} Theorem 2) one can improve this $C^{\infty}$ Uhlenbeck subsequential convergence to convergence of the entire flow line.
\end{proof}


\begin{thebibliography}{s}
%
%
%
%

\bibitem{ADHM} M. Atiyah, V. Drinfeld, N. Hitchin, Y. Manin.,
\emph{Construction
of instantons}, \emph{Phys. Lett}., 65:185--187, 1978.

\bibitem{Aubin} T. Aubin, \emph{Equations diff\'erentielles non lin\'eaires et
Probleme de Yamabe concernant la courbure scalaire}, J. Math. Pures et appl. 55,
(1976) 269-296.

\bibitem{BPST} A. Belavin, A. Polyakov, A. Schwarz, Y. Tyupkin,
\emph{Pseudoparticle solutions of the Yang--Mills equations}, \emph{Phys.
Lett}.,
59B:8--87, 1975

\bibitem{Besse}  A. L. Besse, \emph{Einstein manifolds}. Ergebnisse der
Mathematik und ihrer Grenzgebiete (3) [Results in Mathematics and Related Areas
(3)], 10. Springer-Verlag, Berlin, 1987. xii+510 pp.

\bibitem{Bor} G. Bor, \emph{Yang--Mills fields which are not self-dual} Comm.
Math. Phys. 145, 393-410 (1992).

\bibitem{BW} J.P. Bourguignon, \emph{Formules de Weitzenb\"ock en dimension 4}, in ``G\'eometrie Riemannienne de dimension $4$,'' CEDIC, Paris 1981.

\bibitem{BL} J.P. Bourguignon, H. Lawson, \emph{Stability and isolation
phenomena for Yang--Mills fields} Comm. Math. Phys. 79, 189-230 (1981).

\bibitem{ChenShen} Y-M. Chen, C-L. Shen, \emph{Evolution of Yang-Mills connections}, Differential geometry (Shanghair, 1991), World Sci. Publ., River Edge, NJ (1993), 33-41.

\bibitem{MOD} J. Dodziuk, M. Min-Oo, \emph{An $L_2$-isolation theorem for Yang-Mills fields over complete manifolds}, Compositio Math. 47 (1982), 165-169.

\bibitem{Donaldson1} S.K. Donaldson, \emph{An application of gauge theory to
four-dimensional topology}, J. Diff. Geom. 18 (2), 279-315.

\bibitem{Donaldson2} S.K. Donaldson, \emph{Polynomial invariants for smooth
four-manifolds}, Topology 29 (1990), 257-315.

\bibitem{DK} S.K. Donaldson, P.B. Kronheimer, \emph{The geometry of
four-manifolds}, Oxford Mathematical Monographs, (1990)

\bibitem{Feehan} P. Feehan, \emph{Energy gap for Yang--Mills connections, I:
Four-dimensional closed Riemannian manifolds}, Adv. Math. 296 (2016), 55-84.

\bibitem{FeehanBook} P. Feehan \emph{Global existence and convergence of
solutions to gradient systems and applications to Yang--Mills gradient flow},
\textsf{arXiv:1409.1525}.

\bibitem{Gerhardt} C. Gerhardt, \emph{An energy gap theorem for Yang-Mills connections}, Comm. Math. Phys. 298 (2010), 515-522.

\bibitem{G1}  M. J. Gursky, \emph{Four-manifolds with $\delta W^{+} = 0$ and
Einstein constants of the sphere}, Math. Ann. 318 (2000), no. 3, 417–-431.

\bibitem{GL1}  M. J. Gursky and C. LeBrun, \emph{Yamabe invariants and spin-c
structures}, Geom. Funct. Anal. 8 (1998), no. 6, 965–-977.

\bibitem{GL2} M. J. Gursky and C. LeBrun, \emph{On Einstein manifolds of
positive sectional curvature}, Ann. Global Anal. Geom. 17 (1999), no. 4,
315–-328.

\bibitem{KMN} H. Kozono, Y. Maeda, H. Naito, \emph{Global solution for the Yang--Mills gradient flow on $4$-manifolds}, Nagoya Math. J. 139 (1995), 93-128.

\bibitem{LeBrun} C. LeBrun, \emph{Ricci curvature, minimal volumes, and
Seiberg-Witten theory}, Invent. Math. 145 (2001), no. 2, 279–-316.

\bibitem{LeeParker} J. Lee, T. Parker, \emph{The Yamabe problem}, Bull. Amer. Math. Soc., Vol. 17, No. 1 (1987), 37-91.

\bibitem{MinOo} Min-Oo \emph{An $L_2$-isolation theorem for Yang--Mills fields},
Comp. Math. Vol 47, Fasc. 2, 1982, 153-163.

\bibitem{Parker} T. Parker, \emph{Non-minimal Yang--Mills fields and dynamics}
Invent. Math. (1992), Vol. 107, Issue 2, 397-420.

\bibitem{Parker2} T. Parker, \emph{Gauge theories on four-dimensional Riemannian manifolds}, Comm. Math. Phys. 85 (1982), 563-602.

\bibitem{Rade} J. R\aa de \emph{Decay estimates for Yang--Mills fields: two new
proofs} Global analysis in modern mathematics (Orono, ME, 1991) 91-105, Publish
or Perish, Houston, TX, 1993.

\bibitem{Rade2} J. R\aa de \emph{On the Yang--Mills heat equation in two and three dimensions} J. reine angew. Math. 431 (1992), 123-163.

\bibitem{SS1} L. Sadun, J. Segert, \emph{Non-self-dual Yang--Mills connections
with quadropole symmetry}, Comm. Math. Phys. 145, 363-391 (1992).

\bibitem{Seaman} W. Seaman, \emph{Harmonic two-forms in four dimensions}, Proc.
Amer. Math. Soc, Vol. 112, No. 2, 1991.

\bibitem{Schlatter} A. Schlatter, \emph{Long-time behaviour of the Yang--Mills
flow in four dimensions}, Ann. Glob. Anal. Geom. 15: 1-25, 1997.

\bibitem{Shen} C.L. Shen, \emph{The gap phenomena of Yang-Mills fields over the complete manifold}, Math. Z. 180 (1982), 69-77.

\bibitem{SSU} L.M. Sibner, R.J. Sibner, K. Uhlenbeck, \emph{Solutions to
Yang--Mills equations that are not self-dual}, Proc. Natl. Acad. Sci. USA VOl.
86, 8610-8613, Nov. 1989.

\bibitem{Simon} L. Simon, \emph{Asymptotics for a class of nonlinear evolution
equations, with applications to geometric problems}, Ann. Math. (2) 118 (1983),
525-571.

\bibitem{SW} E.M. Stein, G. Weiss, \emph{Introduction to Fourier analysis on
Euclidean spaces}, Princeton Math. Series 32, Princeton University Press,
Princeton, NJ 1971.

\bibitem{Struwe} M. Struwe, \emph{The Yang--Mills flow in four dimensions},
Calc.
Var. 2, 123-150 (1994).

\bibitem{Xin} Y.L. Xin, \emph{Remarks on gap phenomena in four dimensions}, Calc. Var. PDE 2 (1994), 123-150.

\bibitem{Yang} B. Yang, \emph{The uniqueness of tangent cones for Yang--Mills connections with isolated singularities}, Adv. Math. 180, Issue 2, 648-691.

\end{thebibliography}
\end{document}